\documentclass[3pt]{article}

\usepackage{amsmath,amssymb,amsfonts,bbm,amsthm}
\usepackage{color}
\usepackage{pdfpages}
\usepackage{graphicx}
\usepackage{geometry}
\usepackage{url}
\usepackage{enumerate}
\usepackage{extarrows}
\usepackage{extarrows,booktabs,multirow,diagbox}

\theoremstyle{plain}
\newtheorem{thm}{\bf Theorem}[section]

\newtheorem{lem}[thm]{\bf Lemma}

\theoremstyle{remark}

\newtheorem{rem}{Remark}

\newcommand{\Rg}       {{\hbox{I\kern-.22em\hbox{R}}}}
\newcommand{\Pg}       {{\hbox{I\kern-.22em\hbox{P}}}}
\newcommand{\Eg}       {{\hbox{I\kern-.22em\hbox{E}}}}

\title{Controlled Mean-Reverting Estimation for The AR(1) Model with Stationary Gaussian Noise}

\author{
Chunhao Cai\\
School of Mathematics, Shanghai University of Finance and Economics, Shanghai, China. \\
Email: caichunhao@mail.shufe.edu.cn}
\date{\today}

\begin{document}
\maketitle

\begin{abstract}
This paper deals with the maximum likelihood estimator for the mean-reverting parameter of a first order autoregressive models with exogenous variables, which are stationary Gaussian noises (Colored noise). Using the method of the Laplace transform, both the asymptotic properties and the asymptotic design problem of the  maximum likelihood estimator are investigated. The numerical simulation results confirm the theoretical analysis and show that the proposed maximum likelihood estimator performs well in finite sample.\\
\vspace{3mm}


\end{abstract}

\section{Introduction}\label{section 1}
Temporal dependence in volatility has been one of the most studied problems in financial econometrics. For example,  Gatheral et \textit{a.l.} \cite{GJR18} has presented that the volatility is rough with some real data and statistical tools. Comte and Renault \cite{CR98} also takes the following expression:
\begin{eqnarray*}
dy_t&=&\sigma^*e^{h_t/2}dW_t\\
dh_t&=&\gamma h_tdt+v^*dB_t^H
\end{eqnarray*}
where $y_t$ is the log price of an asset at time $t$, $W_t$ is a standard Brownian motion and $B_t^H$ is a fractional Brownian motion whose covariance function is 
$$
\mathbf{E}(B_t^H B_s^H)=\frac{1}{2}(t^{2H}+s^{2H}-|t-s|^{2H}),\,\, s,t\in [0,T].
$$
If we apply the Euler approximation with discrete time $t=\Delta,\,2\Delta,\, \cdots,\,N\Delta(=T) $ to the volatility $h_t$, then the volatility $h_t$ has the discrete-time model
$$
h_{i\Delta}=\beta h_{(i-1)\Delta}+v\eta_{i\Delta}^H.
$$
where $\beta=1+\gamma \Delta$, $\eta_{i\Delta}^H=B_{i\Delta}^H-B_{(i-1)\Delta}^H$ is a fractional Gaussian noise with distance $\Delta$. If $v=\Delta=1$, $\gamma$ is a negative constant such that $|\beta|<1$ then it is a AR(1) process with fractional Gaussian noise. Regarding that the fractional Gaussian noise is stationary Gaussian noise, we can extend this model and rewrite it as 
\begin{equation}\label{eq: or model}
X_n=\vartheta X_{n-1}+\xi_n,\, |\vartheta|<1
\end{equation}
and $\xi=(\xi_n,\,n\in \mathbb{Z})$ is the centred regular stationary Gaussian noise with
\begin{equation}\label{spectral density}
\int_{-\pi}^{\pi}|\log f_{\xi}(\lambda)|d\lambda<\infty\,,
\end{equation}
where $f_{\xi}(\lambda)$ is the spectral density of $\xi$. Brouste \textit{et a.l.} found the MLE for the unknown parameter $\vartheta$ and proved its consistency. 

In the sense of continuous observation, Brouste and Cai \cite{BC13} considered the controlled drift estimation problem in the fractional O-U process:
\begin{equation}\label{eq: control OU}
dX_t=-\vartheta X_tdt+u(t)dt+dB_t^H,\, \vartheta>0,\, t\in [0,T]
\end{equation}
where $B_t^H$ is a fractional Brownian motion with Hurst parameter $H\in (0,1)$ and $u(t)$ is a deterministic function in a control space. They have found a function in the control space which  maximize the Fisher Information and also obtained the consistency of the MLE of $\vartheta$. 

Still applying the Euler approximation in the model \eqref{eq: control OU} and extend it to the general stationary Gaussian noise, in this paper we will consider a similar model 
\begin{equation}\label{model}
X_n=\vartheta X_{n-1}+u(n)+\xi_n,\,0< |\vartheta|<1, X_0=0\,,
\end{equation}
where $u(n)$ denotes a deterministic function of $n$ and $\xi_n$ is the same noise defined in \eqref{eq: or model}. When considering the problem of estimating the unknown parameter $\vartheta$ with the observation data $X^{(N)}=\left(X_n,\, n=1,\cdots, N\right)$, we denote $L(\vartheta, X^{(N)})$  the likelihood function for $\vartheta$, then the Fisher information can be written as
$$
 \mathcal{I}_N(\vartheta,u)=-\mathbf{E}_{\vartheta}\frac{\partial^2}{\partial \vartheta^2}\ln L(\vartheta,X^{(N)}).
$$
Let $\mathcal{U}_N$ be some control space defined in \eqref{eq: contral space}, therefor we want to find the function such that 
$$
\mathcal{J}_N(\vartheta)=\sup_{u\in \mathcal{U}_N}\mathcal{I}_N(\vartheta,u)\,.
$$
and then find an adapted estimator $\bar{\vartheta}_N$ of the parameter $\vartheta$, which is asymptotically efficient  in the sense that, for any compact $\mathbb{K}\in (0,1)=\{\vartheta|0<\vartheta<1\}$,
\begin{equation}\label{efficiency}
\sup_{\vartheta \in \mathbb{K}}\mathcal{J}_N(\vartheta)\mathbf{E}_{\vartheta}(\bar{\vartheta}_N-\vartheta)^2=1+o(1),
\end{equation}
as $N\rightarrow \infty$.

\begin{rem}
This model is different from the ARX model of Ljung (\cite{Ljung}, Page 73). We present here a direct generalization of \cite{BCK} and \cite{BC13}. We can see that the control space $\mathcal{U}_N$ of \eqref{eq: contral space} is so nearby of that in \cite{BC13}.
\end{rem}

\begin{rem}
In this paper, we suppose that the covariance structure of the noise $\xi$ is known. In fact, if this covariance depends only on one parameter, for example, the Hurst parameter H of fractional Gaussian noise presented in Section \ref{simulation} we can estimate this parameter with the log-peridogram method \cite{Robinson95} or with generalized quadratic variation \cite{IL97} and study the plug-in estimator. This will be our future study.
\end{rem}

\begin{rem}
In this paper, we will not estimate the function $u(n)$ but just find one function which maximize the Fisher Information.
\end{rem}

The organization of this paper is as follows. In Section \ref{section 2}, we give some basic result of the Regular stationary noise $\xi_n$, find the likelihood function and the formula of the Fisher Information. Section \ref{section 3} provides the main results of this paper and Section \ref{simulation} shows some simulation examples to show the performance of the proposed MLE.  All the proofs are collected in the Appendix.

\section{Preliminaries and Notations}\label{section 2}

\subsection{Stationary Gaussian sequences}
Let us suppose that the covariance of the real valued centered stationary Gaussian sequence $\xi=\left(\xi_n\right)_{n\geq 1}$ 
\begin{equation}\label{eq: rho}
\mathbf{E}\xi_m\xi_n=c(m,n)=\rho(|n-m|),\, \rho(0)=1\,,
\end{equation}
is positive defined, then there exist an associated innovation sequence $\left(\sigma_n\varepsilon_n \right)_{n\geq 1}$ where $\varepsilon_n \sim \mathcal{N}(0,1)$, $n\geq 1$ are independent, defined by the following relations:
$$
\sigma_n \varepsilon_n=\xi_1,\,\ \sigma_n \varepsilon_n=\xi_n-\mathbf{E}(\xi_n|\xi_1,\xi_2,\dots, \xi_{n-1}),\, n\geq 2.
$$
It follows from the theorem of Normal Correlation (Theorem 13.1, \cite{Liptser}) that there exists a deterministic kernel $k=(k(n,m),\,n\geq 1,\, m\leq n)$ such that $k(n,n)=1$ and 
\begin{equation}\label{sigma}
\sigma_n\varepsilon_n=\sum_{m=1}^nk(n,m)\xi_m.
\end{equation}
For $n\geq 1$, we will denote by $\beta_{n-1}$ the partial correlation coefficient
\begin{equation}\label{beta}
\beta_{n-1}=-k(n,1).
\end{equation}
The following relationship between $k(\cdot,\,\cdot)$, the covariance function $\rho(\cdot)$ defined in \eqref{eq: rho}, the sequence of partial correlation coefficients $(\beta)_{n\geq 1}$ and the variances of innovation $(\sigma_n^2)_{n\geq 1}$ (see Levinson-Durbin algorithm \cite{Durbin})
\begin{equation}\label{rela 1}
\sigma_n^2=\prod_{m=1}^{n-1}(1-\beta_m^2),\, n\geq 2,\, \sigma_1=1,
\end{equation}
\begin{equation}\label{rela 2}
\sum_{m=1}^n k(n,m)\rho(m)=\beta_n\sigma_n^2,
\end{equation}
\begin{equation}\label{rela 3}
k(n+1, n+1-m)=k(n,n-m)-\beta_n k(n,m)
\end{equation}
Since we assume the positive definiteness of the covariance $c(\cdot, \cdot)$, there also esists an inverse deterministic kernel $K=(K(n,m),\, n\geq 1,\, m\leq n)$ such that 
\begin{equation}\label{eq: Inverse}
\xi_n=\sum_{m=1}^n K(n,m)\sigma_m\varepsilon_m.
\end{equation}\
The relationship of the kernel $k$ and $K$ can be found in \cite{BCK}.

\begin{rem}
It is worth mentioning that the condition \eqref{spectral density} implies that 
$$
\sum_{n\geq 1} \beta_n^2<\infty.
$$
This condition ies theoretically verified for classical ARMA noises. There has been no explicit form of the 
partial autocorrelation coefficients for the fractional Gaussian noise but we can numerically verify this condition. For the very similar fractional ARIMA processes, it has been proved that $\beta_n=O(1)$ in \cite{Hosking81}.
\end{rem}

\subsection{Model Transformation}
Let us define the process $Z=(Z_n,\, n\geq 1)$ such that
\begin{equation}\label{eq:Z}
Z_n=\sum_{m=1}^nk(n,m) X_m,\, n\geq 1\,,
\end{equation}
where $k(n,m)$ is the kernel defined in \eqref{sigma}. Similar to \eqref{eq: Inverse}, we have
$$
X_n=\sum_{m=1}^n K(n,m)Z_m\,.
$$

It is worth mentioning that the process $Z$ has the same filtration of $X$. In the following parts, let the observation be $(Z_1,\, Z_2,\,\cdots,\, Z_N)$. Actually, it was shown in \cite{BCK} the process $Z$ can be considered as the first component of a $2$-dimensional AR(1) process $\zeta=\zeta_n,\, n\geq 1$, which is defined by:
$$
\zeta_n=\left(\begin{array}{c}Z_n \\ \sum\limits_{r=1}^{n-1}\beta_rZ_r    \end{array}\right)\,.
$$

It is not hard to obtain that $\zeta_n$ is a $2$-dimensional Markov process which satisfies the following equation:
\begin{equation}\label{eq:zeta}
\zeta_n=A_{n-1}\zeta_{n-1}+bv(n)+b\sigma_n\epsilon_n,\, n\geq 1,\, \zeta_0=\mathbf{0}_{2\times 1}\,,
\end{equation}
with
\begin{equation}\label{eq:An and b}
A_n=\left(\begin{array}{cc}\vartheta & \vartheta\beta_n \\\beta_n & 1\end{array}\right),\, b=\left(\begin{array}{c}1 \\0\end{array}\right)
\end{equation}
and $\epsilon_n\sim\mathcal{N}(0,1)$ are independent. Following from the idea of \cite{BC13} we will define the control space $\mathcal{V}_N$ of the function $v(n)$:
$$
\mathcal{V}_N=\left\{v\big| \frac{1}{N}\sum_{n=1}^N\left\|\frac{v(n)}{\sigma_{n+1}}  \right\|^2\leq 1   \right\}\,.
$$
From the control space of $\mathcal{V}_N$ we can define that for the function $u(n)$:
\begin{equation}\label{eq: contral space}
\mathcal{U}_N=\left\{u(n)\big| \frac{1}{N}\sum_{n=1}^{N}\left\|\frac{\sum\limits_{m=1}^nk(n,m)u(m)}{\sigma_{n+1}}  \right\|^2\leq 1  \right\}\,,
\end{equation}

\subsection{Fisher Information}
As we have interpreted, the observation will be the first component of the process $\zeta=(\zeta_n,\, n\geq 1)$. Now from the equation \eqref{eq:zeta}, it is easy to write the Likelihood function $L(\vartheta,X^{(N)})$, which depends on the function $v(n)$:
\begin{equation}\label{likelihood function}
L(\vartheta,X^{(N)})=\prod_{n=1}^N\frac{1}{\sqrt{2\pi\sigma_n^2}}
\exp\left(-\frac{1}{2}\sum_{n=1}^N
\left(\frac{b^*\left(\zeta_n-A_{n-1}^{\vartheta}\zeta_{n-1}\right)}
{\sigma_n}\right)^2\right).
\end{equation}

Consequently, the Fisher Information $\mathcal{I}_N(\vartheta, v)$ can be written as
\begin{equation}\label{Fisher Information 1}
\mathcal{I}_N(\vartheta,v)=-\mathbf{E}_{\vartheta}\frac{\partial^2}{\partial \vartheta^2}\ln L(\vartheta, X^{(N)})=\mathbf{E}_{\vartheta}\sum_{n=1}^{N-1}\left(\frac{a_{n}^*\zeta_{n}}{\sigma_{n+1}}\right)^2,
\end{equation}
where $a_n=\left(\begin{array}{c}1 \\\beta_n\end{array}\right)$.

\section{Main Results}\label{section 3}
In this part, we will present our main Results in this paper. First of all, from the presentation of the Fisher Information \eqref{Fisher Information 1} we have 
\begin{thm}\label{optimal input}
The asymptotical optimal input in the class of control $\mathcal{U}_T$ is $u_{opt}(n)=\sum\limits_{m=1}^nK(n,m)\sigma_{m+1}$ for $0<\vartheta<1$ and $u_{opt}(n)=(-1)^n\sum\limits_{m=1}^nK(n,m)\sigma_{m+1}$ or $u_{opt}(n)=(-1)^{n+1}\sum\limits_{m=1}^nK(n,m)\sigma_{m+1}$ for $-1<\vartheta<0$. Moreover,
$$
\lim_{N\rightarrow \infty}\frac{\mathcal{J}_N(\vartheta)}{N}=\mathcal{I}(\vartheta)\,,
$$
where $\mathcal{I}(\vartheta)=\frac{1}{1-\vartheta^2}+\frac{1}{(1-\vartheta)^2}$ for $0<\vartheta<1$ and $\mathcal{I}(\vartheta)=\frac{1}{1-\vartheta^2}+\frac{1}{1+\vartheta}^2$.
\end{thm}

\begin{rem}\label{rem2}
The theorem \ref{optimal input} can be generalized to the AR(p) case with the norm of the matrix of Fisher Information, but this purpose is not so clear as well as as AR(1) that is: when the Fisher Information is larger, the error will be smaller. For this reason, we will illustrate only the result of first order but not order p.
\end{rem}

From Theorem \ref{optimal input}, since the optimal input does not depend on the unknown parameter $\vartheta$, we can consider $\bar{\vartheta}$ as the MLE $\hat{\vartheta}_N$. The following theorem states that $\hat{\vartheta}_N$ will reach the efficiency of \eqref{efficiency}.

\begin{thm}\label{asymptotical normal}
With the optimal input $u_{opt}(n)$ defined in Theorem \ref{optimal input}, for $0<|\vartheta|<1$, the MLE $\hat{\vartheta}_N$ has the following properties:
\begin{itemize}
\item $\hat{\vartheta}_N$ is strong consistency, that is $\hat{\vartheta}_N\overset{a.s.}{\rightarrow }\vartheta$ as $N\rightarrow\infty$.

\item $\hat{\vartheta}_N$ is uniformly consistent on compact $\mathbb{K}\subset \mathbb{R}$, \textit{i.e.} for any $\nu>0$
$$
\lim_{N\rightarrow \infty}\sup_{\vartheta \in \mathbb{K}}\mathbf{P}_{\vartheta}^N\left\{\left|\hat{\vartheta}_N-\vartheta\right|>\nu\right\}=0
$$

\item $\hat{\vartheta}_N$ is uniformly on compacts asymptotically normal, i.e., as $N\rightarrow\infty$,
$$
\lim_{N\rightarrow \infty}\sup_{\vartheta\in \mathbb{K}}\left|\mathbf{E}_{\vartheta}f\left(\sqrt{N}\left(\hat{\vartheta}_N-\vartheta\right)-\mathbf{E}f(\xi)\right)\right|=0,\, \forall f\in \mathcal{C}_b\,,
$$
where $\xi$ is a zero mean Gaussian random variable with variance $\mathcal{I}^{-1}(\vartheta)$ defined in Theorem \ref{optimal input}.
Moreover, we have the uniform on $\vartheta\in \mathbb{K}$ convergence of the moments: for any $q>0$,
$$
\lim_{N\rightarrow \infty}\sup_{\vartheta\in \mathbb{K}}\left|\mathbf{E}_{\vartheta}\left|\sqrt{N}\left(\hat{\vartheta}_N-\vartheta\right)\right|^q -\mathbf{E}|\xi|^q                   \right|.
$$
\end{itemize}
\end{thm}

\section{Simulation study}\label{simulation}
In this section, Monte-Carlo simulations are done for the asymptotical normality of the MLE $\hat{\vartheta}_N$ with different Gaussian noise such as AR(1), MA(1) and fractional Gaussian noise(fGn). However, $\beta_n$ defined in \eqref{beta} for the ARMA case is explicit, we can easily obtain the result of MLE, so we just simulate the fGn with $H=0.65$ and different $\vartheta$ positive and negative.

The covariance function of fGn is 
$$
\rho(|n-m|)=\frac{1}{2}(|m-n+1|^{2H}-2|m-n|^{2H}+|m-n-1|^{2H})
$$
for a known Hurst exponent $H\in (0,1)$. For this simulation we use Wood and Chan method in \cite{Yajima85} and the asymptotical normality of the estimation error $\hat{\vartheta}_N-\vartheta$ follows

\begin{center}
\resizebox{80mm}{40mm}{\includegraphics{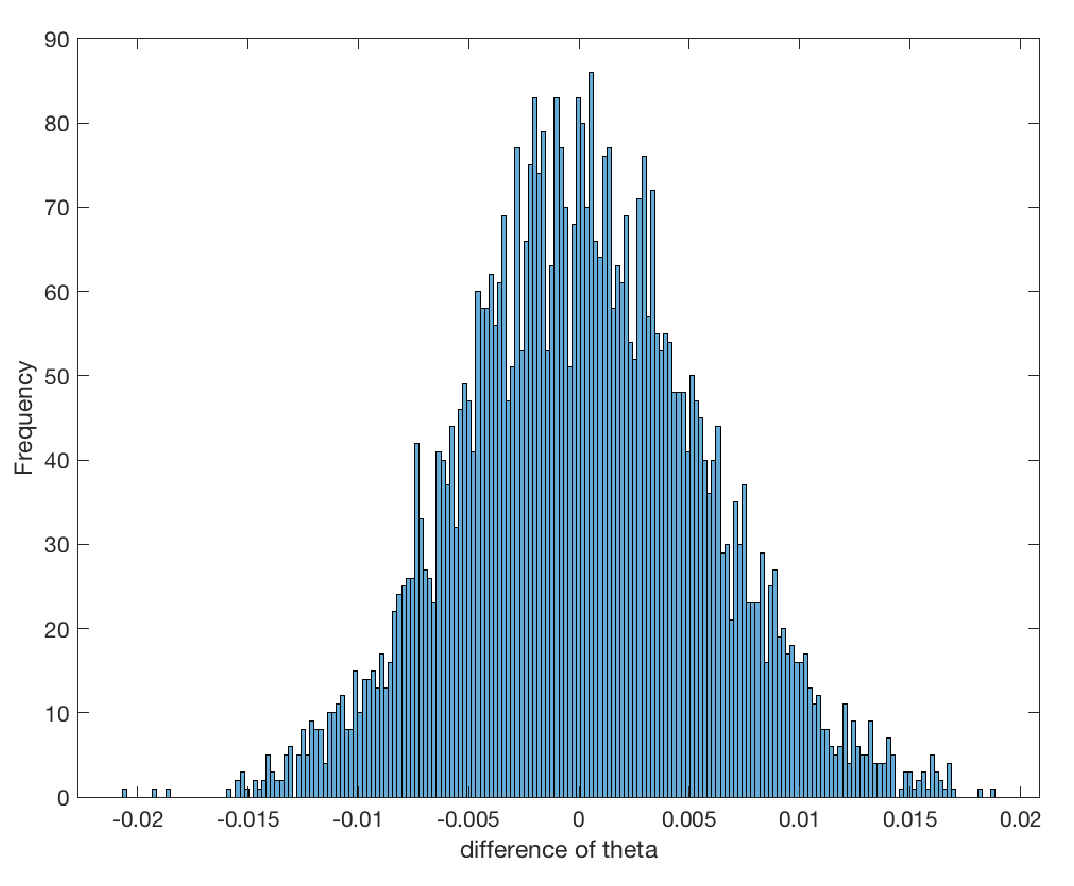}}\\
Fig.1. Histogram of the statistic $\Phi(N,\vartheta,X)$ with $\vartheta=-0.7$.
\end{center}

\begin{center}
\resizebox{80mm}{40mm}{\includegraphics{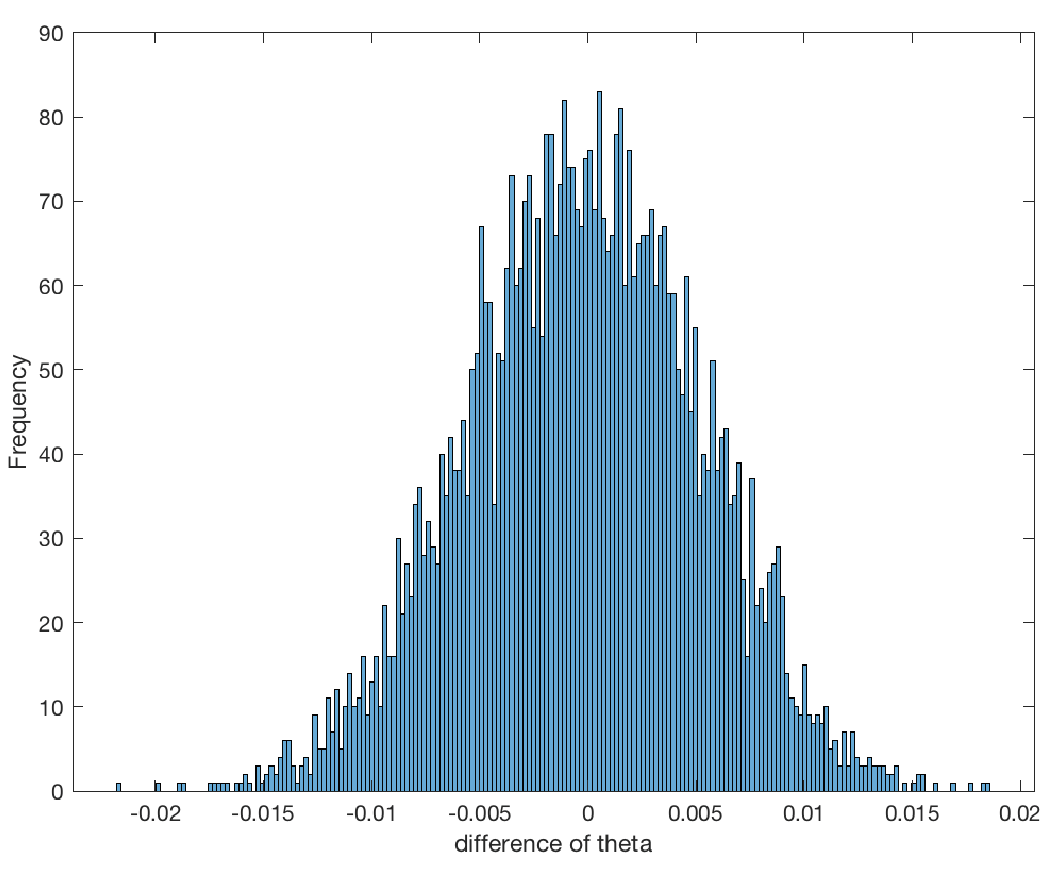}}\\
Fig.2. Histogram of the statistic $\Phi(N,\vartheta,X)$ with $\vartheta=0.7$.
\end{center}

\begin{rem}
In this simulation, the function $u_{opt}(n)$ is just a deterministic function. Its contribution here will be the value of Fisher Information but not others.
\end{rem}

\section{Appendix}
The appendix provides the proofs of Theorem \ref{optimal input} and Theorem \ref{asymptotical normal}. Without special note we only consider of $0<\vartheta<1$, for $-1<\vartheta<0$ the proofs will be the same.
\subsection{Proof of Theorem \ref{optimal input} }
To prove the theorem \ref{optimal input}, we separate the Fisher Information of \eqref{Fisher Information 1} into two parts:
\begin{eqnarray*}
\mathcal{I}_N(\vartheta,v) &=& \mathbf{E}_{\vartheta}\sum_{n=1}^{N-1} \left(\frac{a_{n}^*\zeta_{n}-a_{n}^*\mathbf{E}_{\vartheta}\zeta_{n}+a_{n}^*\mathbf{E}_{\vartheta}\zeta_{n}}{\sigma_{n+1}}\right)^2 \\
&=&\mathbf{E}_{\vartheta}\sum_{n=1}^{N-1} \left(\frac{a_{n}^*(\zeta_{n}-\mathbf{E}_{\vartheta}\zeta_{n})}{\sigma_{n+1}}   \right)^2+\sum_{n=1}^{N-1}\left(\frac{a_{n}^*\mathbf{E}_{\vartheta}\zeta_{n}}{\sigma_{n+1}}    \right)^2\\
&=&\mathbf{E}_{\vartheta}\sum_{n=1}^{N-1} \left(\frac{a_{n}^*\mathcal{P}^{\vartheta}_n}{\sigma_{n+1}}   \right)^2+\sum_{n=1}^{N-1}\left(\frac{a_{n}^*\mathbf{E}_{\vartheta}\zeta_{n}}{\sigma_{n+1}}    \right)^2\\
&=&\mathcal{I}_{1,N}(\vartheta)+\mathcal{I}_{2,N}(\vartheta,v).
\end{eqnarray*}
where $\mathcal{P}_n^{\vartheta}$ satisfies the following equation:
$$
\mathcal{P}^{\vartheta}_n=A_{n-1}^{\vartheta}\mathcal{P}^{\vartheta}_{n-1}+b\sigma_n\epsilon_n,\, \mathcal{P}^{\vartheta}_0=\mathbf{0}_{2\times 1}.
$$
Obviously, $\mathcal{I}_{1}(\vartheta)$ does not depends on $v(n)$. Thus, and as presented in \cite{BCK}, we have 
$$
\lim_{N\rightarrow \infty}\mathbf{E}_{\vartheta}\exp\left(-\frac{1}{2N}\sum_{n=1}^{N-1}\frac{(a_n^*\mathcal{P}^{\vartheta}_n)^2}{\sigma_{n+1}^2} \right)=\exp\left(-\frac{1}{2}\mathcal{I}_1(\vartheta)\right)
$$
and $\mathcal{I}_1(\vartheta)=\frac{1}{1-\vartheta^2}$.

A standard calculation yields
\begin{equation}\label{I1}
\lim_{N\rightarrow \infty}\frac{\mathcal{I}_{1,N}(\vartheta)}{N}=\mathcal{I}_1(\vartheta)=\frac{1}{1-\vartheta^2}.
\end{equation}

To compute $\mathcal{I}_{2,N}(\vartheta)$, let $s(n)=\frac{\mathbf{E}_{\vartheta}\zeta_n}{\sigma_{n+1}}$. Then, we can see that $s(n)$ satisfies the following equation:
\begin{equation}\label{eq:s}
s(n)=A_{n-1}s(n-1)\frac{\sigma_n}{\sigma_{n+1}}+bf(n)\,,
\end{equation}
where $f(n)=\frac{v(n)}{\sigma_{n+1}}$ and it is bounded.

Note that $\beta_n\rightarrow 0$ and
$\frac{\sigma_n}{\sigma_{n+1}}\rightarrow 1$, we assume that for $n=1,2,\cdots$, $\frac{\sigma_n}{\sigma_{n+1}}\leq (1+\varepsilon)$ and $\beta_n\leq \varepsilon$ for the sufficiently small positive constant $\varepsilon$ and $(1+\varepsilon)\vartheta<1$. Consequently, we can state the following result.

\begin{lem}Let $Y=(Y_n,\, n\geq 1)$ be the 2-dimension equation, which  satisfies the following equation:
$$
Y_n=\left(\begin{array}{cc}\vartheta & 0 \\0 & 1\end{array}\right)Y_{n-1}+bf(n),\, Y(0)=y_0\,.
$$

Then, we have
$$
\lim_{N \rightarrow \infty}\frac{1}{N}\sum_{n=1}^N(a_n^*s(n))^2=\lim_{N\rightarrow \infty}\frac{1}{N}\sum_{n=1}^N (b^*Y_n)^2\,.
$$

\begin{proof}
For the sake of notational simplicity, we introduce a 2-dimensional equation $Y=(Y'_n,\, n\geq 1)$, which satisfies the following equation:
$$
Y'_n=\left(\begin{array}{cc}\vartheta & 0 \\0 & 1\end{array}\right)Y_{n-1}\frac{\sigma_n}{\sigma_{n+1}}+bf(n),\, Y'_0=y'_0\,.
$$

In this situation, we have three comparison. First, we compare $b^*s(n)-b^*Y'_n$. A standard calculation implies
$$
s(n)-Y'_n=\left(\begin{array}{cc}0 & \vartheta \beta_{n-1} \\\beta_{n-1} & 0\end{array}\right)(s(n-1)-Y'_{n-1})\frac{\sigma_n}{\sigma_{n+1}},
$$
which implies $b^*(s(n)-Y'_n)\rightarrow 0$.

Now, we compare $b^*Y_n-b^*Y'_n$. A simple calculation shows that
$$
b^*(Y'_n-Y_n)=\vartheta(b^*(Y'_n-Y_n))+\vartheta\left(\frac{\sigma_n}{\sigma_{n+1}}-1\right)b^*Y'_{n-1}\,,
$$
which implies $b^*(Y'_n-Y_n)\rightarrow 0$ since $f(n)$ is bounded and $\frac{\sigma_n}{\sigma_{n+1}}-1\rightarrow 0$.

Finally, as $\beta_n\rightarrow 0$ and the component of $s(n)$ is bounded, we can easily obtain $a_n^*s(n)-b^*s(n)\rightarrow 0$, which achieves the proof.
\end{proof}
\end{lem}

Now, we define $\alpha(n)=b^*Y_n$. Then $\alpha(n)=\vartheta \alpha(n-1)+f(n)$, where  $f(n)$ is in the space of $\mathcal{F}_N=\left\{f(n)\big|\frac{1}{N}\sum\limits_{n=1}^Nf^2(n)\leq 1\right\}$. Since the initial value $\alpha(0)$ will not change our result, we assume $\alpha(0)=0$ without loss of generality.

Let $\mathcal{J}_{2,N}=\sup\limits_{v\in \mathcal{V}_N}\mathcal{I}_{2,N}(\vartheta,v)$. Then, it is clear that
\begin{equation}\label{J2N}
\lim_{N\rightarrow\infty}   \frac{\mathcal{J}_{2,N}(\vartheta)}{N-1}=\lim_{N\rightarrow \infty}\frac{1}{N}\sup\limits_{f\in \mathcal{F}_N}\sum_{n=1}^N\alpha^2(n).
\end{equation}
Now to prove theorem \ref{optimal input}, we only need the following lemma.

\begin{lem}
As $\mathcal{J}_{2,N}(\vartheta)$ presented in \eqref{J2N}, we have
\begin{equation}\label{J2}
\lim_{N\rightarrow \infty}\frac{\mathcal{J}_{2,N}(\vartheta)}{N-1}=\mathcal{I}_2(\vartheta)\,,
\end{equation}
where $\mathcal{I}_2(\vartheta)=\frac{1}{(1-\vartheta)^2}$.
\end{lem}
\begin{proof}
First of all, taking $f(n)=1$, it is easy to get the lower bound
$$
\lim_{N\rightarrow \infty}\frac{1}{N}\sum_{n=1}^N\alpha^2(n)\geq \frac{1}{(1-\vartheta)^2}.
$$

Moreover, a simple calculation shows that
$$
\alpha(n)=\varphi(n)\sum\limits_{i=1}^{n}\varphi^{-1}(i)f(i),\, n\geq 1,\, \alpha(0)=0\,,
 $$
where
$$
\varphi(n)=\vartheta \varphi(n-1), \, \varphi(0)=1\,.
$$

Obviously, we can rewrite $\frac{1}{N}\sum_{n=1}^N\alpha^2(n)$ as
$$
\frac{1}{N}\sum_{n=1}^N\alpha^2(n)=\sum_{n=1}^{N}\left(\varphi(n)\sum_{i=1}^n\varphi^{-1}(i)\frac{f(i)}{\sqrt{N}}\right)\left(\varphi(n)\sum_{i=1}^n\varphi^{-1}(i)\frac{f(i)}{\sqrt{N}}\right).
$$
or
$$
\frac{1}{N}\sum_{n=1}^N\alpha^2(n)=\sum_{i=1}^{N}\sum_{j=1}^{N}F_{N}(i,j)\frac{f(i)}{\sqrt{N}}\frac{f(j)}{\sqrt{N}},
$$
where
$$
F_{N}(i,j)=\sum_{\ell=i\bigvee j}^{N}\left(\varphi(\ell)\varphi^{-1}(i)\right)\left(\varphi(\ell)\varphi^{-1}(j)\right)\,.
$$

Let $\phi_n=\varphi^{-1}(n)\sum\limits_{\ell=n}^{N}\varphi_{\ell}\epsilon_{\ell}$ with $\epsilon_{\ell}\sim\mathcal{N}(0,1)$ are independent. Then, we have
$$
F_{N}(i,j)=\mathbf{E}(\phi_i\phi_j)
$$
and
$$
\phi_{n-1}=\vartheta \phi_n+\epsilon_{n-1},\, \phi_{N}=0.
$$

Let us mentioned that $F_{N}(i,j)$ is a compact symmetric operator for fixed $N$. We should estimate the spectral gap (the first eigenvalue $\nu_1(N)$) of the operator. The estimation of the spectral gap is based on the Laplace transform of $\sum\limits_{i=1}^{N}\phi_i^2$, which is written as
$$
L_{N}(a)=\mathbf{E}_{\vartheta}\exp \left(-\frac{a}{2}\sum_{i=1}^{N}\phi_i^2\right)\,,
$$
for sufficiently small negative $a<0$. On one hand, when $a>-\frac{2}{\nu_1(N)}$, $\phi$ is a centred Gaussian process with covariance operator $F_{N}$. Using Mercer's theorem and Parseval's identity, $L_T(a)$ can be represented by
\begin{equation}\label{Laplace transform 1}
L_N(a)=\prod_{i \geq 1}(1+a\nu_i(N))^{-1/2}\,,
\end{equation}
where $\nu_i(N)$ is the sequence of positive eigenvalues of the covariance operator.  A straightforward algebraic calculation shows
\begin{equation}\label{Laplace transform 2}
L_N(a)=\left(\vartheta^{N-1}\Psi_{N}^1\right)^{-1/2}\,,
\end{equation}
where
$$
\Psi_{N}=\left(\begin{array}{cc}1 & 0\end{array}\right)\left(\begin{array}{cc}\frac{1}{\vartheta} & \frac{1}{\vartheta} \\\frac{a}{\vartheta} & \frac{a}{\vartheta}+\vartheta\end{array}\right)^{N-1}\left(\begin{array}{c}1 \\0\end{array}\right).
$$
For
$$
\Delta=\left(\frac{1+a}{\vartheta}+\vartheta \right)^2-4\geq 0\,,
$$
there exists two real eigenvalues $\lambda_2\geq 1,\, \lambda_1\leq 1$ of the matrix
$$
\left(\begin{array}{cc}\frac{1}{\vartheta} & \frac{1}{\vartheta} \\\frac{a}{\vartheta} & \frac{a}{\vartheta}+\vartheta\end{array}\right).
$$

Then, we can see that
$$
\Psi_N=\left(\frac{\lambda^{N-1}-\lambda^{N-1}}{\vartheta^2}-\frac{\lambda_2^{N-2}-\lambda_1^{N-2}}{\vartheta}\right)\frac{\vartheta}{\lambda_2-\lambda_1}\geq 0\,.
$$

That is to say for $\vartheta>0$ and for any $0>a\geq -(1-\vartheta)^2$, $L_N(a)\geq 0$. Thus, $\lim\limits_{N\rightarrow \infty}\nu_1(N)\leq \frac{1}{(1-\vartheta)^2}$ and we complete the proof.
\end{proof}

\begin{rem}
For $-1<\vartheta<0$, $\Delta\geq 0$ means $\frac{1+a}{\vartheta}+\vartheta\leq -2$ and $0>a>-(1+\vartheta)^2$. As a consequence, we have $\nu_1(N)\leq (1+\vartheta)^2$.
\end{rem}

\subsection{Proof of theorem \ref{asymptotical normal}}
Let $v_{opt}(n)=\sigma_{n+1}$ and $\zeta^o=(\zeta^o_n,\, n\geq 1)$ be the process $\zeta$ with the function $v_{opt(n)}$. Then, we have
$$
\zeta^o_n=A_{n-1}\zeta^o_{n-1}+bv_{opt}(n)+b\sigma_n\epsilon_n,\,\,\, \zeta_0=\mathbf{0}_{2\times 1}
$$

To estimate the parameter $\vartheta$ from the observations $\zeta_1, \zeta_2,\,\cdots,\, \zeta_N$, we can write the MLE of $\vartheta$ uing 
\eqref{likelihood function}
\begin{equation}\label{eq:MLE}
\hat{\vartheta}_N=\left(\sum_{n=1}^N\left(\frac{a_n^*\zeta_n}{\sigma_{n+1}}\right)^2 \right)^{-1}\left(\sum_{n=1}^N\frac{a_n^*\zeta_nb^*\zeta_{n+1}}{\sigma_{n+1}^2}  \right)\,.
\end{equation}

A standard calculation yields
$$
\hat{\vartheta}_N-\vartheta=\frac{M_N}{\langle M\rangle_N}\,,
$$
where
$$
M_N=\sum_{n=1}^N\frac{a_n^*\zeta_n}{\sigma_{n+1}}\epsilon_{n+1},\,\,\,\, \langle M\rangle_N=\sum_{n+1}^N\left(\frac{a_n^*\zeta_n}{\sigma_{n+1}}\right)^2\,.
$$

The second and third conclusion about the asymptotically normality of theorem \ref{asymptotical normal} are crucially based on the asymptotical study of the Laplace transform
$$
\mathcal{L}_N^{\vartheta}\left(\frac{\mu}{N}\right)=\mathbf{E}_{\vartheta}\exp\left(-\frac{\mu}{2N}\langle M\rangle_N \right)\,,
$$
for $N\rightarrow \infty$.

First, we can rewrite $\mathcal{L}^{\vartheta}_N\left(\frac{\mu}{N}\right)$ by the following formula:
$$
\mathcal{L}^{\vartheta}_N\left(\frac{\mu}{N}\right)=\mathbf{E}_{\vartheta}\exp\left(-\frac{1}{2}\sum_{n=1}^N\zeta_n^*\mathcal{M}_n\zeta_n   \right)\,,
$$
where $\mathcal{M}_n=\frac{\mu}{N\sigma_{n+1}^2}a_na_n*$.

As presented in \cite{BCK} and using the Cameron-Martin formula \cite{KLV02}, we have the following result.

\begin{lem}\label{formula Laplace}
For any $N$, the following equality holds:
$$
\mathcal{L}^{\vartheta}_N\left(\frac{\mu}{N}\right)=\prod_{n=1}^N[\det(\mathbf{Id}+\gamma(n,n)\mathcal{M}_n)]^{-1/2}\exp\left(-\frac{1}{2}\sum_{n=1}^Nz_n^*\mathcal{M}_n (\mathbf{Id}+\gamma(n,n)\mathcal{M}_n)^{-1}z_n         \right)\,,
$$
where $(\gamma(n,m),\, 1\leq m\leq n)$ is the unique solution of the equation
\begin{equation}\label{eq:gamma n m}
\gamma(n,m)=\left[\prod_{r=m+1}^n A_{r-1}(\mathbf{Id}+\gamma(r,r)\mathcal{M}_r)^{-1}\right]\gamma(m,m),
\end{equation}
 and the function $(\gamma(n,n,\,n\geq 1))$ is the solution of the Ricatti equation:
\begin{equation}\label{eq:gamma}
\gamma(n,n)=A_{n-1}(\mathbf{Id}+\gamma(n-1,n-1)\mathcal{M}_{n-1})^{-1}\gamma(n-1,n-1)A_{n-1}^*+\sigma_n^2bb^*.
\end{equation}

It is worth to emphasize that $(z_n,\, 1\leq n\leq N)$ is the unique solution of the equation
$$
z_n=m_n-\sum_{r=1}^{n-1}\gamma(n,r)[\mathbf{Id}+\gamma(r,r)\mathcal{M}_r]^{-1}\mathcal{M}_rz_r,\,z_0=m_0.
$$
where $m_n=\mathbf{E}\zeta^o_n$.
\end{lem}

With the explicit formula of the Laplace transform presented in Lemma \ref{formula Laplace}, we have its asymptotical property.

\begin{lem}\label{le:conv laplace}
Under the condition \eqref{spectral density}, for any $\mu\in \mathbb{R}$, we have
\begin{equation}\label{conv laplace}
\lim_{N\rightarrow \infty}\mathcal{L}^{\vartheta}_N\left(\frac{\mu}{N}\right)=\exp\left(-\frac{\mu}{2}\mathcal{I}(\vartheta)\right)
\end{equation}
where $\mathcal{I}(\vartheta)=\frac{1}{1-\vartheta^2}+\frac{1}{(1-\vartheta)^2}$.
\end{lem}

\begin{proof}
In \cite{BCK} we have stated that 
\begin{equation}\label{part1}
\lim_{N\rightarrow \infty} \prod_{n=1}^N[\det(\mathbf{Id}+\gamma(n,n)\mathcal{M}_n)]^{-1/2}=\frac{\mu}{1-\vartheta^2}.
\end{equation}

Since the component of $\gamma(n,n)$ is bounded, we have 
$$
\lim_{N\rightarrow \infty} \mathbf{Id}+\gamma(n,n)\mathcal{M}_n=\mathbf{Id}.
$$

On the other hand, 
\begin{equation}\label{part2}
\sum_{n=1}^Nm_n^*\mathcal{M}_nm_n=\frac{\mu}{N}\sum_{n=1}^N\left(\frac{a_n^*\mathbf{E}\zeta^0_n}{\sigma_{n+1}}\right)^2\longrightarrow \frac{\mu}{(1-\vartheta)^2},\, N\rightarrow \infty.
\end{equation}
which has been presented in last part. At last, notice
$$
\sum_{r=1}^{n-1}\gamma(n,r)[\mathbf{Id}+\gamma(r,r)\mathcal{M}_r]^{-1}\mathcal{M}_rz_r=\sum_{r=1}^{n-1}\left[\prod_{\tau=r+1}^nA_{\tau-1}(\mathbf{Id}+\gamma(\tau,\tau)\mathcal{M}_{\tau})^{-1}   \right][\mathbf{Id}+\gamma(r,r)\mathcal{M}_r]\mathcal{M}_rz_r
$$
we have that
\begin{equation}\label{conv 0}
\lim_{N\rightarrow \infty}\sum_{r=1}^{n-1}\gamma(n,r)[\mathbf{Id}+\gamma(r,r)\mathcal{M}_r]^{-1}\mathcal{M}_rz_r=0.
\end{equation}
Combining \eqref{part1} and \eqref{part2} and \eqref{conv 0}, the Lemma  \ref{le:conv laplace} achieves.
\end{proof}

From this conclusion, it is immediate that
$$
\mathbf{P}_{\vartheta}\lim_{N\rightarrow \infty}\frac{1}{N}\langle M\rangle_N=\mathcal{I}(\vartheta)\,.
$$

Moreover, using the central limit theorem for martingale, we have
$$
\frac{1}{\sqrt{N}}M_N\Longrightarrow \mathcal{N}(0,\mathcal{I}(\vartheta)).
$$

Consequently, the asymptotical part of the theorem \ref{asymptotical normal} is obtained.

The strong consistent is immediate when we change $\frac{\mu}{N}$ with a positive proper constant $\mu$ in Lemma \ref{le:conv laplace} because the determinant part tends to $0$ as  presented in section 5.2 of \cite{BCK}and the extra part is bounded.

\paragraph{Acknowledgements} The authors would like to thank the editor and reviewers for their valuable comments which improve the manuscript.

\paragraph{Availability of data and materials} Data sharing not applicable to this article as no data sets were generated or analyzed during the current study.

\paragraph{Funding} Not applicable.

\paragraph{Competing interests} The authors declare that they have no competing interests.

\end{document}